\documentclass[11pt]{article} 
 \usepackage{authblk}

\usepackage{amsmath,amsthm,amssymb} 
\usepackage{ amssymb }
\usepackage{tikz-cd}
\usepackage{enumitem}
\usepackage[a4paper, total={6.5in, 8in}]{geometry}
\usepackage{breqn}

\usepackage{tikz}
\usepackage[inline]{asymptote}

\usepackage{caption}
\newtheorem{lemma}{Lemma}[section]
\newtheorem{theorem}[lemma]{Theorem}
\newtheorem{theorem*}{Theorem}
\newtheorem{corollary}[lemma]{Corollary}

\newtheorem{conjecture}[lemma]{Conjecture}
\theoremstyle{definition}

\numberwithin{equation}{section}

\begin{document}
 
 \title{Defining Curvature as a Measure via Gauss-Bonnet on Certain Singular Surfaces}

\author[  ]{Robert S Strichartz}

\affil[  ]{Mathematics Department, Cornell University, Ithaca, NY 14853, str@math.cornell.edu}

\maketitle

\begin{abstract}
 We show how to define curvature as a measure using the Gauss-Bonnet Theorem on a family of singular surfaces obtained by gluing together smooth surfaces along boundary curves. We find an explicit formula for the curvature measure as a sum of three types of measures: absolutely continuous measures, measures supported on singular curves, and discrete measures supported on singular points. We discuss the spectral asymptotics of the Laplacian on these surfaces. 
 
 \end{abstract}

\section{Introduction}

 Curvature on a surface $M$ with Riemannian metric $g$ is traditionally defined as a pointwise function $K_g(x)$ on $M$, or the interior of $M$ if $M$ has a boundary $\partial M$. However, the classical Gauss-Bonnet Theorem says 
 \begin{equation}
 \int_N K_g dA + \int_{\partial N} \kappa ds = 2 \pi \chi (N) 
 \end{equation}
 for any smooth subsurface $N$ of $M$, where $dA$ is the Riemannian area element, $ds$ is the arclength element on $\partial N$, $\kappa$ is the one-dimensional curvature function on $\partial N$, and $\chi(N)$ is the Euler characteristic. We see here that the absolutely continuous signed measure 
 \begin{equation}
 \mu_{K} (N) = \int_N K_g dA 
 \end{equation}
 is the key ingredient. We can always obtain the pointwise curvature $K_g(x)$ from the measure curvature $\mu_{K}$, so we really don't lose any information by identifying curvature as a measure. We call this the measure-centered viewpoint. Notice that the Gauss-Bonnet Theorem also adopts this viewpoint for the curvature of the boundary $\partial N$. 
 
 This is philosophy, not mathematics. The point of this paper is that this philosophy leads to some interesting mathematics. Specifically, we would like to extend the definition of curvature to more general spaces than smooth Riemannian surfaces so that an analog of Gauss-Bonnet continuous to hold. We can already point to a simple example that is well-known. Let $M$ be a convex polyhedron, with vertices $V$. On $M \backslash V$, we have a flat Riemannian metric, so we only see curvature at the vertices. It is natural to think of the curvature of $M$ as a measure supported on $V$, a sum of delta masses multiplied by the angle defect $2 \pi - \sum \alpha_j,$ where $\alpha_j$ are the angles of the faces that come together at $v \in V$. (In the case of a polyhedron that is  not convex, the same formula is valid and yields a signed measure.) And the analog of Gauss-Bonnet is simply 
 \begin{equation} 
 \chi(M) = \sum_{v \in V}( 2 \pi - \sum \alpha_j).
 \end{equation}
 
 In this paper we concentrate on surfaces with $1$-dimensional singularities obtained by gluing together smooth surfaces with smooth boundaries along the boundaries. In the simplest example we take $M$ to be $M_1 \sqcup M_2$ factored by an identification map $\varphi: \partial M_1 \to \partial M_2$ that is one-to-one, onto, and smooth. 
 
 \begin{theorem}
 It is possible to define a curvature measure on $M$ using a version of the Gauss-Bonnet Theorem such that 
 \begin{equation} \label{eq:1.4}
 \mu_{ K }(M) = \mu_{K}(M_1) + \mu_{ K }(M_2) + \nu_1 + \nu_2  
 \end{equation}
 where $\nu_1$ and $\nu_2$ are supported on the identified boundary, with 
 \begin{equation}
 \nu_j = \kappa ds \text{ on } \partial M_j \text{ for } j=1,2. 
 \end{equation}
 \end{theorem}

 We give the proof of the Theorem in section 2, along with some straightforward generalizations. In section 3, we discuss a broader class of examples in which the curvature measure is allowed to have both one-dimensional and discrete parts. (A work in progress \cite{DPS} describes an example of a convex surface whose curvature measure is a fractal measure.) In section 4 we discuss the spectral asymptotics of a natural Laplacian on these surfaces. 
 
 According to the standard philosophy of mathematics, this paper seems to be putting the cart before the horse. After all, shouldn't we try to describe a category of objects that might be called $\textit{singular surfaces}$, and then give a definition of curvature measures on all singular surfaces, proving a version of Gauss-Bonnet? I agree that this would be far better than what I am presenting here. But since I don't know how to do this, it will have to wait for the future. And don't forget Hamlet's famous comment to Horatio!
 
 There are, of course, many other theories of objects that might be called singular surfaces; for example, Alexandrov spaces \cite{Sh}. Some, but not all of our surfaces are Alexandrov spaces, and the kind of curvature associated with Alexandrov spaces is qualitative, not quantitative. 
 
 An entirely different approach to understanding the curvature measure on the identified boundary is to examine the asymptotics of the area of $D_r(x)$ as $r \to 0$ for points $x$ on the boundary, where $D_r(x)$ is the disc of radius $r$ centered at $x$ in $M$. Clearly $D_r(x) = D^1_r(x_1) \cup D^2_r(x_2)$ where $D^j_r(x_j)$ is the disc of radius $r$ centered at $x_j$ in $M_j$, and $x_j$ are the boundary points in $M_j$ identified at $x$. But we know 
 \begin{equation} 
 \text{Area } D^j_r(x_j) = \frac{ \pi }{2} r^2 + \frac{1}{3} \kappa(x_j) r^3 + O(r^4) .
 \end{equation}
 If we agree that the curvature measure along the identified boundary should be $\rho (x)ds$ where $\rho (x)$ satisfies 
 \begin{equation}
  \text{Area }D_r(x) =  \pi r^2 + \frac{1}{3} \kappa(x_j) r^3 + O(r^4) 
\end{equation}
then we obtain (\ref{eq:1.4}). See p. 35 of \cite{H} for an example of this approach.

 \section{ Proof of Theorem}
\begin{proof}
 We assume that $M_j$ are $C^2$ surfaces with $C^2$ boundaries, with Riemannian metrics $g_j$ that are $C^2$, and the gluing identification map $\varphi: \partial M_1 \to \partial M_2$ is also $C^2$. 

To simplify the exposition we first assume that the surfaces $(M_j, g_j)$ are flat. Let $x$ and $y$ denote points on the identified boundaries, so $x_j$ and $y_j$ are the corresponding points on $\partial M_j$, with $x_2 = \varphi(x_1)$ and $y_2 = \varphi (y_1)$. We assume that $x$ and $y$ are sufficiently close that the construction in Figure 2.1 is possible.

\begin{figure}[h!] 
\begin{tikzpicture}

\draw (0,7) .. controls (5,6) and (2,1) .. (5,0.5);
\draw  (0.5,4.6) -- (2.15,6);
\draw  (1,2.3) -- (3.45,2.4);
\draw (0.5,4.6) --(1,2.3);

\draw  (2.15,5.5) -- (2.4,5.7);
\draw (2.15,5.5) -- (1.89,5.75);

\draw  (2.95,2.76) -- (3.4,2.8);
\draw  (2.95,2.76) -- (3,2.4);


\draw (0.55,4.3) arc (-80:35:3.5mm);

\draw (1.3,2.3) arc (0:95:3.5mm);

\draw (0.3,7.3) node 
{ $\partial M_1$ };

\draw (-0.5,5) node 
{ $ M_1$ };
\draw (3.5,4.3) node 
{ $ A_1$ };
\draw (2.2,3.8) node 
{ $ N_{1}$ };

\draw (2.6,6.2) node 
{ $ x_1$ };
\draw (3.9,2.35) node 
{ $ y_1$ };

\draw (1.2,5.7) node 
{ $ L_{11}$ };
\draw (2.2,2) node 
{ $ L_{12}$ };
\draw (0.4,3.5) node 
{ $ L_{1}$ };

\draw (1.2,4.5) node 
{ $ a_{11}$ };
\draw (1.6,2.7) node 
{ $ a_{12}$ };

\draw (9,7) .. controls (8,6) and (10,5)   .. (10,4);
\draw (10,4) .. controls (10,3) and (9,2)   .. (12,1);

\draw  (8.9,5.8) -- (10.7,7.0);
\draw  (9.85,2.6) -- (11.6,2.8);
\draw (10.7,7.0) -- (11.6,2.8);

\draw  (9.1,5.5) -- (9.4, 5.68);
\draw (9.42, 5.68) --(9.22, 5.98);

\draw  (9.85,2.88) --(10.1, 2.93);
\draw (10.1, 2.93) -- (10.15, 2.63);

\draw (10.3,6.75) arc (-155:-58:3.5mm);
\draw (11.55,3.15) arc (90:185:3.5mm);

\draw (12.5,5.5) node 
{ $ M_{2}$ };
\draw (9,7.4) node 
{ $\partial M_2$ };
\draw (10.5,5) node 
{ $ N_{2}$ };

\draw (9.6,6.8) node 
{ $ L_{21}$ };
\draw (11.5,5) node 
{ $ L_{2}$ };
\draw (11,2.3) node 
{ $ L_{22}$ };

\draw (8.6,5.7)  node 
{ $ x_{2}$ };
\draw (9.5,2.7)  node 
{ $ y_{2}$ };

\draw (9.6,4.1)  node 
{ $ A_{2}$ };

\draw (10.45,6.3) node
{ $a_{21}$};
\draw (11.1,3.3) node
{ $a_{22}$};

\end{tikzpicture}

\caption{Figure 2.1}
\label{Figure 2.1}
\end{figure}

Here $L_{11}$ and $L_{12}$ are line segments in $M_1$ perpendicular to $\partial M_1$ at $x_1$ and $y_1$, and similarly $L_{21}$ and $L_{22}$ are line segments perpendicular to $\partial M_2$ at $x_2$ and $y_2$. Then we may interpret the unions $L_{11} \cup L_{21}$ and $L_{12} \cup L_{22}$ as geodesic line segments in $M$. It is not significant how long we take the perpendicular line segments, as long as the line segment $L_1$ connecting the endpoints of the perpendiculars $L_{11}$ and $L_{12}$ stay within $M_1$, and similarly for $L_2$ in $M_2$. We denote the arc joining $x$ and $y$ in the identified boundary by $A$, and let $A_1$ and $A_2$ denote the corresponding arcs in $\partial M_1$ and $\partial M_2$. 

Let $Q$ denote the quadrilateral in $M$ whose sides are the straight line segments $L_1$, $L_{12} \cup L_{22}$, $L_2$ and $L_{11} \cup L_{22}$, enclosing the region $N= N_1 \cup N_2$. For the Gauss-Bonnet Theorem to hold for $N$ in $M$, we need to have 

\begin{equation}
\mu_{K}(N) = a_{11} + a_{12} + a_{21} +a_{22} - 2 \pi . \label{eq:2.1}
\end{equation}

Note that the individual angles depend on the choices of the lengths of the perpendicular segments, but elementary geometry shows that $a_{11} + a_{12}$ and $a_{21}+a_{22}$ are independent of these choices. Thus, we may take  \eqref{eq:2.1} as the definition of $\mu_{K} (N)$. Because $M_1$ and $M_2$ are flat, the support of the measure $\mu_{K}$ must be the identified boundary. But now we can apply Gauss-Bonnet to the regions $N_j$ in $M_j$ to obtain 

\begin{equation} \label{eq:2.2}
\begin{split}
a_{11} + a_{12} - \pi & = \nu_1(A) \\
a_{21}+ a_{22} - \pi & = \nu_2(A) 
\end{split}
\end{equation}

 Adding these and comparing with \eqref{eq:2.1} we obtain \eqref{eq:1.4} in the flat case.
 
 In the case of curved surfaces we again use the construction in Figure $2.1$, where the segments $L_j$ and $L_{jk}$ are geodesics. Then in place of \eqref{eq:2.2} we have 
 \begin{equation} \label{eq:2.3}
\begin{split}
a_{11} + a_{12} - \pi & = \nu_1(A)  + \mu_{K}(N_1)\\
a_{21}+ a_{22} - \pi & = \nu_2(A) + \mu_{K}(N_2)
\end{split}
\end{equation}
  and this again leads to \eqref{eq:1.4}. 
  
\end{proof}
  
We note that, although the definition \eqref{eq:2.1} gives the Gauss-Bonnet Theorem for one type of domain in $M$, if we combine it with the Gauss-Bonnet Theorem on domains in $M_1$ and $M_2$ we obtain the result on all domains in $M$. 

It is straightforward to extend the theorem to the setting where $M$ is obtained from a finite collection of surfaces $ \{ M_j \}_{j \in J}$ glued together along components of the boundaries $\{ \partial M_j \}$. In fact we could allow gluing of more than two boundaries together. The resulting object would no longer be a surface, of course. 

The construction leads to measures on the identified boundaries that are absolutely continuous with respect to arclength measure on the boundaries of $M_j$ (because of the assumptions on the identification map the arclength measures on $\partial M_1$ and $\partial M_2$ are mutually absolutely continuous). We can also identify certain subarcs on the boundary where the measure is positive, if both $\partial M_1$ and $\partial M_2$ are positive on the subarc, and similarly for negativity.

  \section{Examples with point singularities} 
  
  In this section we discuss briefly some examples of surfaces with both line and point singularities, again obtained by gluing. We will simplify the discussion to look at examples with just a single point singularity at the intersection of several line singularities, but of course everything is local so with some more gluing we could handle a finite number of point singularities. 
  
  The set-up is shown in Figure 3.1.
  
\begin{figure}[h!] 
\centering 
\begin{tikzpicture}
\draw (5,7) .. controls (4,6) and (7,5.8)   .. (6,4);
\draw (5,7.5) node { $L_{41}$};

\draw (6,4) .. controls (5,3) and (4,1)   .. (6,1);
\draw  (6,0.5) node { $L_{23}$};

\draw (3,4) .. controls (3.2,5) and (4.1, 5) .. (5, 4);
\draw (5,4) .. controls (5.4,3.6) and (5.7, 4.7) .. (6, 4);
\draw  (2.5,4) node { $L_{34}$};

\draw (6,4)--(6.4, 4);
\draw  (6.4,4) .. controls (6.8,4) and (7.5, 3) .. (8.5, 4);
\draw  (9, 4) node { $L_{12}$};

\draw (6.4,4) arc (0:76: 3.5mm);
\draw  (6.6,4.3) node { $\theta_1$};
\draw (7.6,5.3) node { $M_1$};

\draw (5.4,4) arc (180:60: 5mm);
\draw (5.6,4.7)  node{ $\theta_4$};
\draw (4.7,5.4)  node{ $M_4$};

\draw (5.45,4.03) arc (180:270: 3.5mm);
\draw (5.35,3.65) node { $\theta_3$};
\draw (4.35,3.1) node { $M_3$};

\draw (6.3,4) arc (0:-122: 3.5mm);
\draw (6.3,3.5) node { $\theta_2$};
\draw (6.8,2.5) node { $M_2$};
\end{tikzpicture}

\caption{Figure 3.1}
\label{Figure 3.1}
\end{figure}
  
  The curves $L_{j (j+1)}$ are identified pieces of the boundaries of $M_j$ and $M_{j+1}$. The angles $\theta_j$ are between the boundary pieces in $M_j$ of $L_{(j-1)j}$ and $L_{j(j+1)}$. If we delete a small neighborhood of the intersection point, then the method in section 2 shows that the curvature measure has the absolutely continuous part supported in the interior of $M_j$ and the 1-dimensional parts supported on the identified curves $L_{j(j+1)}$ so the only issue is what happens in the neighborhood of the intersection point. 
  
  To be more specific, suppose we have smooth manifolds $M_1, M_2, \ldots M_n$ with boundaries that are smooth except for one corner singularity $y_j$ with angle $\theta_j$ (in Figure $3.1$ we have $n=4$). Let $\gamma_{j1}$ and $\gamma_{j2}$ denote the portions of $\partial M_j$ on either side of $y_j$. Suppose we are given $C^2$ identification maps $\varphi_j : \gamma_{(j-1)2} \to \gamma_{j1}$ with $\varphi_j(y_{j-1}) = y_j$. Then we let $M$ be the union of the $M_j$ modulo the identifications. Note that all the corner points $y_j$ are identified in $M$ in the single intersection point we call $y$. The curves $L_{(j-1)j}$ in Figure 3.1 are the identified boundary pieces $\gamma_{(j-1)2}$ and $\gamma_{j1}$ (of course we use cyclic notation in the variable $j$). 
  
  In Figure 3.2 we zoom in on one of the surfaces $M_j$ in Figure 3.1 and perform a construction analogous to that in Figure 2.1. We choose points $x_{j1}$ on $\gamma_{j1}$ and $x_{j2}$ on $\gamma_{j2}$, and we require that $x_{(j-1)2}$ is identified with $x_{j1}$.

  \begin{figure}[h!] 
\centering 
\begin{tikzpicture}
\draw (0.5,2) .. controls (2,4.2) and (4,4.8)   .. (6,5);
\draw (1.5,3.75) node { $ \gamma_{j1}$};
\draw (0.5,2) .. controls (2,1) and (4,0.5)   .. (5.6,0.5);
\draw (1.5,1.1) node { $ \gamma_{j2}$};

\draw (0.15,2) node { $y_i$ };
\draw (5,3) node {$M_j$};
\draw (2.2,2.8) node {$N_j$};

\draw ( 3,0.9) -- (3.3, 2);
\draw ( 4,3.1) -- (3.6, 4.5);
\draw (3.3, 2) --( 4,3.1);
\draw (3.5, 2.3) arc (50:260: 3mm);
\draw (2.8, 2.2) node { $b_j$ };
\draw ( 3.95,3.35) arc (100:230: 3mm);
\draw ( 3.5,3.2) node { $a_j$ };

\draw ( 2.68,0.95)-- (2.78, 1.25);
\draw (2.78, 1.25) --(3.08, 1.16);
\draw (3, 0.6) node { $x_{j2} $};

\draw (3.7, 4.2) -- (3.4, 4.1);
\draw (3.3, 4.4) -- (3.4, 4.1);
\draw (3.75, 4.75) node { $x_{j1} $};

\draw (0.85,1.8) arc (-40:60: 3.5mm);
\draw (1.2,2) node { $\theta_j$ };

\end{tikzpicture}

\caption{Figure 3.2}
\label{Figure 3.2}
\end{figure}

We also require that the distances of the points $x_{j1}$ and $x_{j2}$ to $y_j$ are bounded above by $\epsilon$, where $\epsilon$ is a parameter that will eventually approach zero. We take geodesic segments perpendicular to $\gamma_{j1}$ at $x_{j1}$ (and similarly for $\gamma_{j2}$ and $x_{j2}$), and connect them up by another geodesic segment at angles $a_j$ and $b_j$. When then obtain a pentagon in $M_j$ with three geodesic sides and the curved sides that are segments of $\gamma_{j1}$ and $\gamma_{j2}$, and let $N_j$ be the subset of $M_j$ bounded by this pentagon. 

Let $N$ be the union of the $N_j$ with the identification of the $\gamma_{(j-1)2}$ and $\gamma_{j1}$ sides. This will be a neighborhood of $y$ in $M$. We note that perpendicular geodesic segments at $x_{j1}$ in $M_j$ and $x_{(j-1)2}$ in $M_{j-1}$ combine to form a single geodesic segment in $M$. Thus, $N$ is a $2n$-gon in $M$ bounded by geodesic segments with interior angles $\{ a_j \}$ and $\{ b_j \}$. If we are to have Gauss-Bonnet hold on $M$ we need to define

\begin{equation}
\mu_{K}(N) = \sum_{j=1}^n(a_j +b_j) -2(n-1) \pi 
\end{equation}

On the other hand, from Gauss-Bonnet on $N_j$ in $M_j$ we have 

\begin{equation}
\mu_{K}(N_j) + \int_{y_j}^{x_{j1}} \kappa (\gamma_{j1}) ds +  \int_{y_j}^{x_{j2}} \kappa (\gamma_{j2}) ds +a_j+b_j+ \theta_j -2 \pi
\end{equation}

We may compute $\mu_{\kappa}( \{ y \} )$ by taking the limit as $\epsilon \to 0$ of $\mu_{K}(N) - \sum_{j} \mu_{\kappa}(N_j)$. Clearly the limits of the curvature integrals over the shrinking line segments will vanish, so 

\begin{equation}
\begin{split}
\mu_{K}( \{ y \} & = \lim_{ \epsilon \to 0} \sum_{j=1}^n (a_j + b_j) - 2(n-1) \pi - \sum_{j=1}^n (a_j +b_j+ \theta_j - 2 \pi) \\
& = ( \sum_{j=1}^n \theta_j) -2 \pi.
\end{split}   
\end{equation}

\begin{corollary}
In this setting we have 
\begin{equation}
\mu_{K}(M) = \sum_{j=1}^n \mu_{K} (M_j) + \sum_{j=1}^n ( \nu_{j1} + \nu_{j2}) + (( \sum_{j=1}^n \theta_j) -2 \pi) \delta_y
\end{equation}
where $\nu_{j1}$ and $\nu_{j2}$ are the curvature measures along $\gamma_{j1}$ and $\gamma_{j2}$ as in the Theorem. 

\end{corollary}

 \section{Spectrum of the Laplacian} 
 It is possible to define a Laplacian on $M$, via the weak formulation $\Delta u =f$ means $u \in dom \mathcal{E} $, $f \in L^2(M)$ and 
 \begin{equation}
 \mathcal{E} (u,v) = - \int_{M} f v dA \text{ for all } v \in dom \mathcal{E} 
 \end{equation}
 
 Here the area measure $dA$ is just the sum of the area measures on $M_j$ and the energy form is the sum of energy forms on $M_j$ 
 \begin{equation} 
 \mathcal{E} (u,v) = \sum_j \int_{M_j} \nabla u \cdot \nabla v dA 
 \end{equation} 
 However, the domain $dom( \mathcal{E})$ of the energy form requires a more careful explanation. On a surface, the finiteness of the energy $\mathcal{E}(u,u)$ does not imply that $u$ is continuous, but the discontinuities cannot be too pervasive. Thus the finiteness of $\int_{M_j} | \nabla u|^2 dA$ implies that there is a well defined trace of $u$ on $\partial M_j$ that belongs to the Sobolev space $H^{1/2}( \partial M_j)$. So $dom( \mathcal{E} )$ is defined to be the functions with $\mathcal{E} (u,u)$ finite and whose traces are equal on the identified boundaries. By considering test functions $v$ supported in the interior of $M_j$, it follows that $\Delta u =f$ means that the pointwise formula holds in the interior of $M_j$. It is the continuity condition on the identified boundaries that requires the above careful formulation. 
 
 Assume $M$ is compact. Then the Laplacian is a negative definite self-adjoint operator with compact resolvent. The spectrum consists of a discrete sequence
 \begin{equation}
 \lambda_1 \le  \lambda_2 \le   \lambda_3 \le \ldots \to \infty
 \end{equation}
 of eigenvalues of $- \Delta$ with eigenfunctions 
 \begin{equation}
 - \Delta u_j = \lambda_j u_j
 \end{equation}
 with $\{ u_j \}$ giving an orthonormal basis of $L^2(M)$ (note that $\lambda_1 =0$ and $u_1$ is constant). The famous paper of Mark Kac \cite{K} discusses the relationships between the spectrum and the geometry of $M$ in the case of smooth manifolds. It is thus a fundamental problem to examine similar relationships for the surfaces described in this paper. 
 
 We consider here the special case of the ``double" of a smooth manifold $M_1$ with boundary $\partial M_1$. That is, we take $M_2$ isometric to $M_1$ and take essentially the identity map to glue the boundaries. Note that we have a natural symmetry $\sigma$ of $M$ that interchanges the isometric points of $M_1$ and $M_2$, and $\sigma$ is the identity on the identified boundaries. 
 
 \begin{lemma} 
 The spectrum of $M$ is simply the union of the Dirichlet and Neumann spectra of $M_1$. A Dirichlet eigenfunction on $M_1$ yields an eigenfunction on $M$ by odd reflection under $\sigma$. Similarly, a Neumann eigenfunction on $M_1$ yields an eigenfunction on $M$ by even reflection under $\sigma$. And there are no other eigenfunction on $M$.
 \end{lemma}
 
 \begin{proof}
 Since the symmetry $\sigma$ commutes with the Laplacian, we can break up the eigenspaces on $M$ into odd and even functions under the symmetry $\sigma$. The odd functions vanish on $\partial M_1$ so they correspond to Dirichlet eigenfunctions on $M_1$. Similarly, the even functions have vanishing normal derivative on $\partial M_1$, so they correspond to Neumann eigenfunctions on $M_1$.
 \end{proof} 
 
 Consider the eigenvalue counting function $N_{M}(t) = \# \{ \lambda_j \le t \}$ on $M$, and similarly the Dirichlet and Neumann eigenvalue counting functions $N_M(t)$ and $N_N (t)$ on $M_1$. The lemma says $N_M(t) = N_D(t) +N_N(t)$. 
 
 Now the Weyl-Ivrii asymptotics \cite{I} on $M_1$ say that 
 
 \begin{equation}
 \begin{split}
 N_D(t) = & \frac{ \text{Area}(M_1)}{4 \pi} t - \frac{ \text{Length}( \partial M_1)}{4 \pi} t^{1/2} + o(t^{1/2}) \\
  N_N(t) = & \frac{ \text{Area}(M_1)}{4 \pi} t + \frac{ \text{Length}( \partial M_1)}{4 \pi} t^{1/2} + o(t^{1/2}), 
  \end{split}
  \end{equation}
so 
\begin{equation}
N_M(t) = \frac{ \text{Area}(M)}{4 \pi} t+ o(t^{1/2})
\end{equation}

According to the conjectures in \cite{St} and \cite{MS} we should do better in the case that $M_1$ has constant curvature if we average the errors. Let 
\begin{equation}
\tilde{N}(t) =  \frac{ \text{Area}(M)}{4 \pi} t + \mu_{K}(M),
\end{equation}
  and 
 \begin{equation}
 A(t) = \frac{1}{t} \int_0^t (N(s) - \tilde{N}(s)) ds
 \end{equation}
 
 \begin{conjecture} 
 We have the average error estimate
 \begin{equation}
 A(t) = O (t^{-1/4}).
 \end{equation}
 \end{conjecture}
 
 There are some simple examples of $M_1$ for which the conjecture is known to be true ( rectangle, equilateral triangle, isosceles right triangle, see \cite{JS} \cite{St}) because we can compute $N_D(t)$ and $N_M(t)$ exactly. On the other hand, if $M$ is a flat disc, the conjecture is open, even though the eigenvalues are given explicitly in terms of zeros of Bessel functions and derivatives of Bessel functions.

 \bibliographystyle{annotate}

 \end{document}